\documentclass[a4paper,12pt,reqno]{amsart}
\usepackage[margin=1in]{geometry}
\usepackage{amsfonts, amsmath, amsthm, amssymb, epsfig, graphicx, multicol, setspace, multirow, xcolor, bbm, enumerate}
\usepackage[pdftex]{hyperref}

\DeclareMathOperator{\lcm}{lcm}

\DeclareMathOperator{\os}{OS}
\DeclareMathOperator{\so}{SO}

\DeclareMathOperator{\sos}{SOS}

\DeclareMathOperator{\V}{V}

\DeclareMathOperator{\per}{per}

\newcommand{\w}{\omega}
\newcommand{\G}{\Gamma}
\newcommand{\N}{\mathbbm{N}}
\newcommand{\Z}{\mathbbm{Z}}
\newcommand{\Q}{\mathbbm{Q}}
\newcommand{\R}{\mathbbm{R}}

\newcommand{\scal}[2]{\operatorname{scal}_{#1}(#2)}
\newcommand{\Scal}[2]{\operatorname{Scal}_{#1}(#2)}
\newcommand{\den}[2]{\operatorname{den}_{#1}(#2)}

\theoremstyle{definition} \newtheorem{thm}{Theorem}[section]
\theoremstyle{definition} 
\theoremstyle{definition} \newtheorem{lem}[thm]{Lemma}
\theoremstyle{definition} \newtheorem{cor}[thm]{Corollary}
\theoremstyle{definition} \newtheorem{prop}[thm]{Proposition}
\theoremstyle{definition} 
\theoremstyle{definition} \newtheorem{ex}[thm]{Example}

\begin{document}
	\author{Jeanine Concepcion H.~Arias and Manuel Joseph C.~Loquias}
	\address{Institute of Mathematics, University of the Philippines Diliman}
	\email[J.C.H.~Arias]{jcarias@math.upd.edu.ph}
	\email[M.J.C.~Loquias]{mjcloquias@math.upd.edu.ph}
	\title{Similarity Isometries of Point Packings}
	
	
	\begin{abstract}
		A linear isometry $R$ of $\R^d$ is called a similarity isometry of a lattice $\Gamma \subseteq \R^d$ 
		if there exists a positive real number $\beta$ such that $\beta R\Gamma$ is a sublattice of (finite index in)~$\Gamma$.
		The set $\beta R\Gamma$ is referred to as a similar sublattice of $\Gamma$.
		A (crystallographic) point packing generated by a lattice $\Gamma$ is a union of $\Gamma$ with finitely many shifted copies of~$\Gamma$. 
		In this study, the notion of similarity isometries is extended to point packings.  
		We provide a characterization for the similarity isometries of point packings and identify the corresponding similar subpackings.  
		Planar examples will be discussed, namely, the $1 \times 2$ rectangular lattice and the hexagonal packing (or honeycomb lattice).
		Finally, we also consider similarity isometries of point packings about points different from the origin by studying similarity isometries of shifted point packings. 
		In particular, similarity isometries of a certain shifted hexagonal packing will be computed and compared with that of the hexagonal packing.
	\end{abstract}
		
	\maketitle
	
	\section{Introduction}
		A linear isometry $R$ of $\R^d$ is called a \emph{similarity isometry} of a lattice $\Gamma \subseteq \R^d$ if 
		there exists a positive real number $\beta$ such that $\beta R\Gamma$ is a sublattice of (finite index in) $\Gamma$.
		The set $\beta R\Gamma$ is referred to as a \emph{similar sublattice} of $\Gamma$.
		Similarity isometries of lattices may be viewed as a generalization of coincidence isometries of lattices: in the latter, 
		the lattice is only rotated while in the former, a uniform scaling factor is further applied to the lattice.	
		Coincidence site lattices and coincidence isometries have been used by crystallographers to geometrically describe 
		the interfaces where crystals and quasicrystals of different orientations meet (called grain boundaries)~\cite{F11,B70}.
		On the other hand, similar sublattices and similarity isometries arise for instance in color symmetries of crystals and quasicrystals~\cite{BG04},
		and in structures that are self-similar (e.g.: fractals and tilings with singularities).
		Both similar and coincidence sublattices have been used in the design of multiple description lattice vector quantizers~\cite{AS10b,AS10,SAV15}.
		From a theoretical standpoint, the group of coincidence isometries and the group of similarity isometries are intimately related~\cite{G11,GB08,Z14}.
		These topics belong to a wider class of combinatorial problems for lattices and $\Z$-modules (for a recent survey, see~\cite{BZ17}).	
		
		A (\emph{crystallographic}) \emph{point packing} generated by a lattice $\Gamma$ is a union of $\Gamma$ with finitely many shifted copies of $\Gamma$. 
		Point packings are precisely the locally finite point sets whose translation group forms a lattice of full rank~\cite{BG13}.
		In the context of the sphere packing problem, they appear as non-lattice periodic packings~\cite{CSB01}.
		In crystallography, they serve as models of ideal crystals, that is, crystals having multiple atoms per primitive unit cell.
		Examples include the hexagonal packing or honeycomb lattice, diamond lattice (crystal structure of diamond, tin, silicon, and germanium), 
		and hexagonal close packing (crystal structure of quartz).  
		
		Point packings appear in the literature under different names.
		For instance, Dolbilin et~al.~referred to point packings in~\cite{DLS98} as ideal or perfect crystals, 
		and gave minimal sufficient geometric conditions on a discrete subset of $\R^d$ to be an ideal crystal.
		Schymura and Yuan used the term discrete lattice-periodic sets for point packings, and they considered the homogeneous and inhomogeneous problem for such sets in~\cite{SY19}.
		The term multilattice has also been used to pertain to a point packing, and arithmetic classifications of multilattices have been studied in~\cite{PZ03,I13}.
		
		In this study, we extend the notion of similarity isometries to point packings.  
		Theorem~\ref{component2} provides a characterization for the similarity isometries of point packings. 
		We also establish in Appendix~\ref{app} that the set of similarity isometries of a point packing forms a monoid under certain conditions.
		Planar examples will be discussed, namely, the $1 \times 2$ rectangular lattice and the hexagonal packing viewed as point packings. 
		Finally, we also examine similarity isometries of point packings about points different from the origin.
		To this end, we consider (linear) similarity isometries of a translated copy of a point packing, or what we call a \emph{shifted point packing}.
		This is because rotating a point packing about the point $-x\in\R^d$ is equivalent to rotating the translate of the point packing by $x$ about the origin. 
		To illustrate this, we examine the similarity isometries of a certain shifted hexagonal packing and compare them with the similarity isometries of the hexagonal packing.	
	
	\section{Preliminaries}
		A discrete subset $\Gamma$ of $\R^d$ is a \emph{lattice} if it is the $\Z$-span of $d$ linearly independent vectors in $\R^d$ over $\R$.  
		As a group, $\Gamma$ is isomorphic to the free abelian group of rank~$d$.  
		A \emph{sublattice} $\Gamma'$ of $\Gamma$ is a subgroup of $\Gamma$ of full rank, that is, with finite index in $\G$. 
		Geometrically, the index of $\G'$ in $\G$ may be viewed as the quotient of the volumes of the fundamental domains of $\G'$ and $\G$.
		A \emph{period} of $\G$ is an element $t \in \R^d$ for which $t+\G = \G$. 
		The \emph{set of periods} of $\G$, denoted by $\per(\G)$, is given by $\per(\G) = \{ t \in \R^d : t + \G = \G\}$.
		
		A linear isometry $R$ of $\R^d$ is called a \emph{similarity isometry} of $\G$ if there exists $\beta \in \R^+$ for which $\beta R\G$ is a sublattice of $\G$.
		The sublattice $\beta R\G$ is referred to as a \emph{similar sublattice} of $\G$.
		The set of similarity isometries of $\G$ forms a group and is denoted by $\os(\G)$~\cite{BGHZ08}.
		We write $\os(\G)= \{ R \in \operatorname{O}(d,\R) : \beta R\G \subseteq \G \text{ for some } \beta \in \R^+ \}$.
		Similarly, the set of similarity rotations of $\G$ is denoted by $\sos(\G) = \os(\G) \cap \so(d)$ and is a subgroup of $\os(\G)$.
		Existence of similar sublattices as well as properties of similarity isometries for particular lattices have been well-studied (see~\cite{CRS99,BM98,H08,BHM08,BSZ11}).
		
		Given a similarity isometry $R$ of $\G$, the \emph{set of scaling factors} $\Scal{\G}{R}$ of $R$ with respect to $\G$ is given by 
		$\Scal{\G}{R} = \{ \beta \in \R : \beta R\G \subseteq \G \}$. 		
		Clearly, $R \in \os(\G)$ if and only if $\Scal{\G}{R} \neq \{0\}$.
		The smallest positive element of $\Scal{\G}{R}$ is called the \emph{denominator} of $R$ with respect to $\G$, denoted by $\den{\G}{R}$. 
		From~\cite{BZ17}, we have 
		$\Scal{\G}{R} = \den{\G}{R}\, \Z = \{ k \cdot \den{\G}{R} : k \in \Z \}$ (see also~\cite{GB08,G11}).
		In addition, for any $R, S \in \os(\G)$, we have $\Scal{\G}{R} \Scal{\G}{S}\subseteq\Scal{\G}{RS}$. 
		
		It is easy to verify that $R\in\os(\G)$ if and only if there exists $\alpha \in \R^+$ for which $\G \cap \alpha R\G$ is a sublattice of $\G$ and $\alpha R\G$. 
		In fact, $\G \cap \alpha R\G$ is a sublattice of $\G$ and $\alpha R\G$ whenever
		$\G \cap \alpha R\G$ is a sublattice of either $\G$ or $\alpha R\G$~\cite{BZ17}. 
		We also define the set of scaling factors $\scal{\G}{R}$ of a similarity isometry $R$ of $\G$ given by
		$\scal{\G}{R} = \{ \alpha \in \R : \G \cap \alpha R\G \subseteq \G \}$.
		Here, $\scal{\G}{R} = \den{\G}{R}\, \Q^{\ast} = \{ k \cdot \den{\G}{R} : k \in \Q^{\ast} \}$, and so $\Scal{\G}{R} \subseteq \scal{\G}{R} \cup \{0\}.$
	
		\begin{ex}\label{ex1}
			Let $\G$ be the square lattice $\Z^2$, which we associate with the ring of Gaussian integers $\Z[i] = \{ a+bi : a,b \in \Z \}$ to simplify computations.
			Take $\beta = \sqrt{5}$ and $R$ to be the rotation about the origin by $\tan^{-1}(2)$ with rotation matrix 
			$\frac{1}{\sqrt{5}}
			\left(\begin{smallmatrix}
				1 & -2 \\
				2 & \ 1
			\end{smallmatrix}\right)$. 
			In the setting of $\Z[i]$, the rotation $R$ corresponds to multiplication by the complex number $(1+2i)/\sqrt{5}$. 
			Then $R$ is a similarity isometry of $\G$ with $\den{\G}{R} = \sqrt{5}$ and $\Scal{\G}{R} = \sqrt{5}\, \Z$. 
			Observe in Figure~\ref{figure1} that $\beta R\G \subseteq \G$.			
			\begin{figure}[!ht]
				\centering
				\includegraphics[width=5 cm]{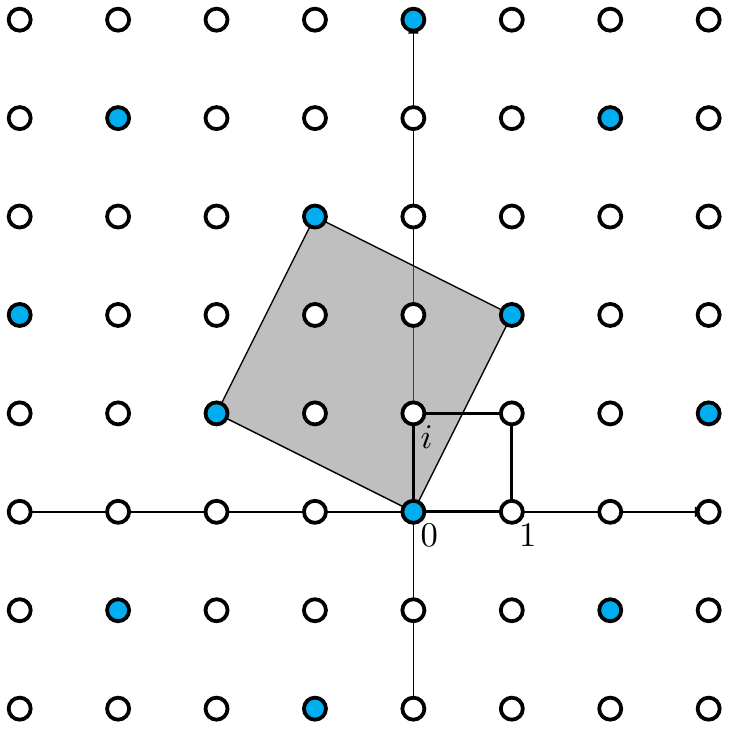}
				\caption{Let $\G$ be the square lattice $\Z[i]$, $\beta = \sqrt{5}$, and $R$ the rotation about the origin by $\tan^{-1}(2)$. 
				The lattice $\beta R\G$ (blue dots) is a similar sublattice of $\G$ that is obtained by applying $R$ to $\G$ (white dots), followed by a scaling factor of $\beta$. }	
				\label{figure1}		
			\end{figure}
		\end{ex}
	
		A subset $L$ of $\R^d$ is said to be a (\emph{crystallographic}) \emph{point packing} generated by the lattice $\G \subseteq \R^d$ 
		if $L$ is the union of $\G$ and a finite number of translated copies of $\G$, that is, 
		$L = \bigcup_{k=0}^{m-1} (x_k + \G)$ where $m \in \N$, $x_0 = 0$, $x_k \in \R^d$ for $1 \leq k \leq m-1$, 
		and $x_{k_1}-x_{k_2} \notin \G$ whenever $k_1 \neq k_2$ (so that the sets $x_k+\G$ are distinct).
		We refer to the lattice $\G$ as the \emph{generating lattice} of $L$.
		The shifted lattices $x_k+\G$ comprising $L$ are called \emph{components} of $L$, and the vector $x_k$ is called a \emph{shift vector} of $L$.
		The set of shift vectors of $L$ will be denoted by $V_L := \{ x_0,x_1,\ldots,x_{m-1}\}$.
		The \emph{set of periods} of $L$, denoted by $\per(L)$, is given by $\per(L) = \{ t \in \R^d : t + L = L\}$.
		
		In general, point packings are not lattices.
		Note also that for a component $x_k+\G$ of~$L$, we can take $x_k'=x_k+\ell$ with $\ell \in \G$ to be another shift vector for $x_k+\G$.
		Thus, without loss of generality, we will choose the shift vectors so that all of them lie within a fundamental domain of $\G$.
		Nonetheless, even with such a choice of shift vectors, the decomposition of a point packing $L$ into a finite union of shifted lattices is not unique.
		This is illustrated by the following example.
	
		\begin{ex} \label{ex-compact}
			Let $L = \G \cup (\frac{1+i}{2}+\G)$ be the point packing generated by the square lattice $\G = \Z[i]$, as shown in Figure~\ref{compact-square}.
			As a point packing generated by $\G$, $L$ has two components, namely, $\G$ and $\frac{1+i}{2}+\G$.
			However, $L$ is also a square lattice obtained by rotating $\G$ about the origin by $\frac{\pi}{4}$ followed by a scaling factor of $\frac{\sqrt{2}}{2}$.
			Writing $L = \G'$, $L$ may now be viewed as a point packing generated by $\G'$, having  only one component.
			Note that $\G \subsetneq \G'$ and $\per(\G) \subsetneq \per(\G')$, since the shift vector $\frac{1+i}{2}$ is contained in $\per(\G')$ but not in $\per(\G)$. 
			\begin{figure}[!ht]
				\centering
				\includegraphics[width=4 cm]{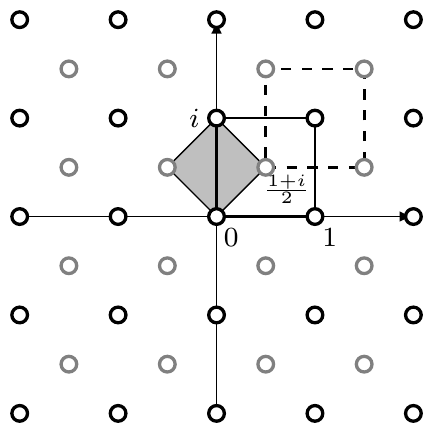}
				\caption{Let $L$ be the union of the square lattice $\G = \Z[i]$ (black dots) and its shifted copy $\frac{1+i}{2}+\G$ (gray dots).
				Here, $L$ is also the square lattice obtained by rotating $\G$ about the origin by $\frac{\pi}{4}$ followed by a scaling factor of $\frac{\sqrt{2}}{2}$. }
				\label{compact-square}
			\end{figure}
		\end{ex}
	
		In general, given a point packing $L$ generated by a lattice $\G$, there exists a generating lattice $\G'$ containing $\G$ such that $\per(\G') = \per(L)$ (\cite{P04}, 
		an alternative proof is given in \cite{EFL}).
		Consequently, when $L$ is generated by such a lattice $\G'$, then $L$ is written with the least number of components.
		Nevertheless, the succeeding results still hold even if $L$ is not expressed in terms of the maximal generating lattice~$\G'$.
		Indeed, in Example~\ref{ex-component} and in Section 4.1, 
		we view the decomposition of a lattice $L$ into cosets of a (proper) sublattice $\G$
		as a point packing generated by $\Gamma$ (with more than one component).
	
	\section{Similarity Isometries of Crystallographic Point Packings}\label{simpointpack}
	
		We now proceed to generalize the notion of a similarity isometry to point packings.
		A point packing $L'$ will be called a \emph{subpacking} of $L$ if $L' \subseteq L$ such that $L'$ has finite index in $L$. 
		By index of $L'$ in $L$, we mean the ratio of the density of points in $L$ by the density of points in $L'$, or equivalently, 
		the ratio of the number of points of $L$ by the number of points of $L'$ in a unit cell of the generating lattice of $L'$.
		Here, we assume that the generating lattices of $L$ and $L'$ have the same rank, and hence the index of $L'$ in $L$ is finite.
		
		A linear isometry $R$ is a \emph{similarity isometry} of $L$ if there exists $\beta \in \R^+$ such that $\beta RL$ is a subpacking of $L$.
		Here, $\beta RL$ is called a \emph{similar subpacking} of $L$.
		The set of similarity isometries of $L$ is denoted by 
		$\os(L) = \{ R \in \operatorname{O}(d,\R) : \beta RL \subseteq L \text{ for some } \beta \in \R^+ \}$.
		The set of similarity rotations of $L$ will be denoted by $\sos(L) := \os(L) \cap \so(d)$.
		Given a similarity isometry $R$ of a point packing $L$, we define the \emph{set of scaling factors} $\Scal{L}{R}$ of $R$ with respect to $L$ 
		by $\Scal{L}{R} = \{ \beta \in \R : \beta RL \subseteq L \}$.
		The smallest positive element of $\Scal{L}{R}$ will be called the \emph{denominator} of $R$ with respect to $L$.
		Observe that $0$ is always in $\Scal{L}{R}$, and we have $R \in \os(L)$ if and only if $\Scal{L}{R} \neq \{0\}$.
		
		We start by describing the intersection of a component of the similar subpacking $\beta RL$ with a component of $L$.
		The following lemma states that such an intersection is either empty or is a shifted copy of $\G \cap \beta R\G$.		
		
		\begin{lem} \label{component1-lem}
		Let $\G$ be a lattice in $\R^d$, $x_j, x_k \in \R^d$, $R \in \operatorname{O}(d, \R)$, and $\beta \in \R^+$.
		\begin{enumerate}[1.]
			\item Then $\beta R(x_k+\G) \cap (x_j + \G)\neq\varnothing$ if and only if $\beta Rx_k-x_j\in\G+\beta R\G$.
		
			\item If $\beta R(x_k+\G) \cap (x_j + \G)\neq\varnothing$, then 
				\begin{alignat}{2}
					\label{eq1}\beta R(x_k+\G) \cap (x_j + \G) &= x_j+\ell+ (\G \cap \beta R\G)\\
					\label{eq2}&= \beta Rx_k + \beta R\ell' + (\G \cap \beta R\G)
				\end{alignat}
				for some $\ell, \ell' \in \G$. 
			\end{enumerate}
		\end{lem}
		
		\begin{proof}
			The proof of item 1 is straightforward.
			
			For item 2, first note that $\beta R(x_k+\G) \cap (x_j + \G) = x_j + [(\beta Rx_k-x_j+\beta R\G)\cap\G]$. 
			Since $\beta R(x_k+\G) \cap (x_j + \G)$ is non-empty, 
			there exist $\ell, \tilde\ell \in \G$ such that $\beta Rx_k-x_j = \ell - \beta R\tilde\ell$.  
			Thus, $(\beta Rx_k-x_j+\beta R\G)\cap\G = (\ell+\beta R\G) \cap \G = \ell + (\G\cap\beta R\G)$. 
			This proves~\eqref{eq1}.
			Equation~\eqref{eq2} can be shown in a similar fashion.
		\end{proof}
		
		The following theorem is a characterization of similarity isometries of a point packing $L$. 
		Essentially, it tells us that $\beta RL$ is a similar subpacking of $L$ if and only if
		every component of $\beta RL$ intersects a fixed number of components of $L$.
		
		\begin{thm} \label{component2}
			Let $L = \bigcup_{k=0}^{m-1} (x_k + \G)$ be a point packing generated by a lattice $\G \subseteq \R^d$, $R \in \operatorname{O}(d, \R)$, and $\beta \in \R^+$. 
			Then $R\in\os(L)$ and $\beta\in\Scal{L}{R}$ if and only if 
			$R \in \os(\G)$ and $\beta \in \scal{\G}{R}$ such that 
			for every $k \in \{0,1,\ldots,m-1\}$, there exist $n$~distinct elements $j_1, \ldots, j_n \in \{0,1,\ldots,m-1\}$, where $n=[\beta R\G : (\G \cap \beta R\G)]$,
			for which $\beta R x_k - x_{j_i} \in \G + \beta R\G$ for all $i \in \{1,\ldots,n\}$.
		\end{thm}

		\begin{proof}
			Suppose that $R\in\os(L)$ and $\beta\in\Scal{L}{R}$.
			Then $\beta RL$ is a subset of $L$ of finite index, and so each component $\beta R(x_k + \G)$ of $\beta RL$ must intersect at least one component of $L$.
			Thus, $J:=\{j \in \{0,1,\ldots,m-1\} : \beta R(x_k+\G) \cap (x_j + \G) \neq \varnothing \}$ is non-empty.
			Note that for each $j\in J$, we have $\beta Rx_k-x_j\in\G+\beta R\G$ by Lemma~\ref{component1-lem}.
			Moreover, by~\eqref{eq2}, there exists $\ell'_j \in \G$ 
			such that $\beta R(x_k+\G) \cap (x_j + \G) = \beta Rx_k + \beta R \ell'_j + (\G \cap \beta R\G)$.
			Since $\beta R(x_k + \G) \subseteq L$, we get 
			\begin{alignat*}{2}
				\beta R (x_k + \G)  &= \beta R(x_k + \G) \cap L \\
				&= \bigcup_{j \in J} [\beta R(x_k+\G) \cap (x_j + \G)] \\
				&= \beta Rx_k + \bigcup_{j \in J} [\beta R \ell'_j + (\G \cap \beta R\G)]. 
			\end{alignat*}
			First, we see that $\beta RL=\bigcup_{k=0}^{m-1}\beta R(x_k+\G)$ consists of shifted copies of $\G\cap \beta R\G$, including $\G\cap \beta R\G$ itself. 
			Since $\beta RL$ is of finite index in $L$ and $m$ is finite, it follows that $\G\cap \beta R\G$ is a sublattice of $\G$.
			Hence, $R\in\os(\G)$ and $\beta\in\scal{\G}{R}$.
			Next, we obtain from above that $\beta R\G = \bigcup_{j \in J}[\beta R \ell'_j + (\G \cap \beta R\G)]$. 
			Consequently, the right-hand side of the equality should be the decomposition of $\beta R\G$ into cosets of $\G \cap \beta R\G$ in $\beta R\G$.
			This means that $|J| = [\beta R\G : (\G \cap \beta R\G)]=n$. Setting $J=\{j_1,\ldots,j_n\}$, the conclusion now follows.
			
			Conversely, Lemma~\ref{component1-lem} implies that for each $i\in \{1,\ldots, n\}$, 
			$\beta R(x_k+\G) \cap (x_{j_i} + \G) = \beta Rx_k + \beta R\ell'_i + (\G \cap \beta R\G)$ for some $\ell'_i \in \G$. 
			Note that $\G \cap \beta R\G$ is a sublattice of $\beta R\G$ because $R\in\os(\G)$ and $\beta\in\scal{\G}{R}$.
			Furthermore, each $\beta R(x_k+\G) \cap (x_{j_i} + \G)$ is distinct and $n=[\beta R\G : (\G \cap \beta R\G)]$, which imply that
			$\{\beta R\ell'_1,\ldots,\beta R\ell'_n\}$ forms a complete set of coset representatives of $\G\cap\beta R\G$ in $\beta R\G$.
			Hence, for every $k \in \{0,1,\ldots,m-1\}$, we have $\beta R(x_k+\G)\subseteq L$ because
			\begin{alignat*}{2}
				\beta R(x_k+\G) \cap L&\supseteq\bigcup_{i=1}^{n}[\beta R(x_k+\G)\cap (x_{j_i} +\G)]\\
				&=\beta Rx_k +\bigcup_{i=1}^{n}[\beta R\ell'_i + (\G \cap \beta R\G)]\\
				&=\beta R(x_k+\G).
			\end{alignat*}					
			Now, Lemma~\ref{component1-lem} also tells us that for each $i\in \{1,\ldots, n\}$,
			$\beta R(x_k+\G) \cap (x_{j_i} + \G)=x_{j_i}+\ell_i+ (\G \cap \beta R\G)$ for some $\ell_i \in \G$. 
			Since $\G\cap\beta R\G$ is also a sublattice of $\G$, we obtain that $\beta R(x_k+\G) \cap (x_{j_i} + \G)$ is of finite index in $x_{j_i}+\G$. 
			We then conclude that $\beta RL=\bigcup_{k=0}^{m-1} \beta R(x_k + \G)$ is a subset of $L$ of finite index.
			Therefore, $R\in\os(L)$ and $\beta\in\Scal{L}{R}$.
		\end{proof}
	
		Theorem~\ref{component2} implies that $\os(L) \subseteq \os(\G)$ and $\Scal{L}{R}\subseteq\scal{\G}{R}$ for any similarity isometry $R\in\os(L)$.
		One can still somehow improve the latter inclusion: if $\beta \in \Scal{L}{R}$ then $\beta\in\frac{1}{n} \Scal{\G}{R}$, where $n = [\beta R\G : (\G \cap \beta R\G)]$. 
		Indeed, we have $n\beta R\G \subseteq \G \cap \beta R\G \subseteq \G$ which means that $n\beta \in \Scal{\G}{R}$. 
		Another result that follows from Theorem~\ref{component2} is that $n$ can be at most the number $m$ of components of $L$.
		
		Let $R\in\os(L)$ and $\beta\in\Scal{L}{R}$. 
		Consider the case when a component of $\beta RL$ intersects more than one component of~$L$. 
		Suppose $k,j_1, j_2 \in \{0,1,\ldots,m-1\}$ such that $j_1\neq j_2$, and $\beta R(x_k+\G) \cap (x_{j_i} + \G)\neq\varnothing$ for $i\in\{1,2\}$. 
		By Lemma~\ref{component1-lem}, we have
		$\beta R x_k - x_{j_i} \in \G + \beta R\G$ for $i \in \{1,2\}$,
		from which we obtain $x_{j_2} - x_{j_1} \in \G + \beta R\G$. 
		Recall from the previous paragraph that if $n = [\beta R\G : (\G \cap \beta R\G)]$ then $n\beta R\G \subseteq \G$. 
		Thus, $\G + \beta R\G \subseteq \G + \frac{1}{n}\G = \frac{1}{n}\G$, and so $x_{j_2} - x_{j_1} \in \frac{1}{n}\G$, where $2 \leq n \leq m$. 
		This yields the following result.  
		
		\begin{cor} \label{component2c}
			Let $L = \bigcup_{k=0}^{m-1} (x_k + \G)$ be a point packing generated by a lattice $\G \subseteq \R^d$, $R \in \os(L)$, $\beta \in \Scal{L}{R}$, 
			and $n = [\beta R\G : (\G \cap \beta R\G)]$.  
			If $2 \leq n \leq m$, then there exist distinct $j_1, j_2 \in \{0,1,\ldots,m-1\}$ such that $x_{j_2} - x_{j_1} \in \frac{1}{n}\G$. 
		\end{cor}
		
		We illustrate the previous results in the following example.  
		
		\begin{ex} \label{ex-component}
			Let $L = \G \cup (1+\G) \cup (2+\G)$ be the square lattice $\Z[i]$ viewed as a point packing that is generated by the $3 \times 1$ rectangular lattice 
			$\G=\{ 3a+bi : a,b \in \Z \}$. 
			Here, $\V_L = \{ 0,1,2 \}$.
			Take $\beta = 1$ and $R$ to be the rotation about the origin by $\frac{\pi}{2}$ in the counterclockwise direction.
			Then $R$ is a similarity isometry of $L$ with $\beta \in \Scal{L}{R}$, and $\beta RL=L$ is a similar subpacking of $L$.
			By inspection, one obtains that $n = [\beta R\G : (\G \cap \beta R\G)] = 3$. 
			Hence, by Theorem~\ref{component2}, each component of $\beta RL$ intersects all three components of $L$, as can be seen in Figure~\ref{figure0}.
			One can also verify that $\beta R x_k - x_j \in \G + \beta R\G = \Z[i]$ for all $x_k, x_j \in \V_L$. 
			Finally, observe that $x_k - x_j \in \{\pm 1, \pm 2\} \subseteq \frac{1}{3}\G$ for all $x_k, x_j \in \V_L$, as expected from Corollary~\ref{component2c}. 
			
			\begin{figure}[!ht]
				\centering
				\includegraphics[width=4 cm]{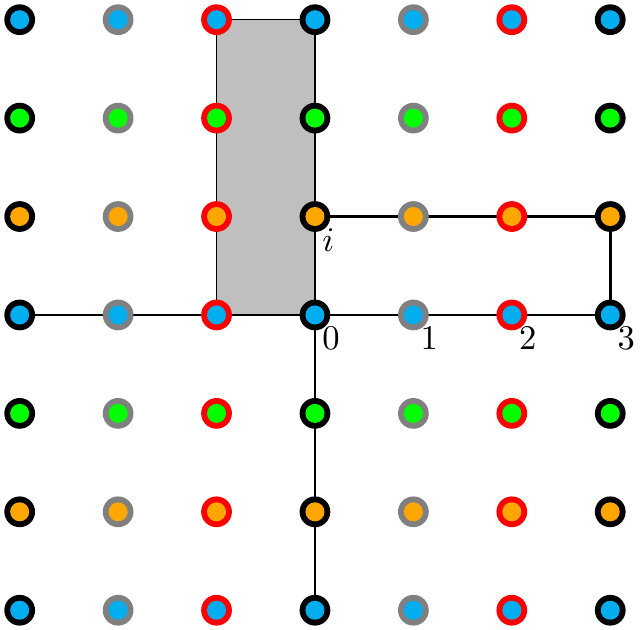}
				\caption{Let $L$ be the square lattice $\Z[i]$ written as the union of the $3 \times 1$ rectangular lattice $\G=\{ 3a+bi : a,b \in \Z \}$ (black outline) 
					with its shifted copies $1+\G$ (gray outline) and $2+\G$ (red outline).
					If $\beta=1$ and $R$ is the rotation about the origin by $\frac{\pi}{2}$, then $R\in\os(L)$ and $\beta\in\Scal{L}{R}$.
					Here, $\beta RL$ consists of the union of $\beta R\G$ (blue dots), $\beta R(1+\G)$ (yellow dots), and $\beta R(2+\G)$ (green dots), 
					and each component of $\beta RL$ intersects $[\beta R\G : (\G \cap \beta R\G)] = 3$ components of $L$.}
				\label{figure0}		
			\end{figure}
		\end{ex}
		
		Given $R\in\os(L)$ and $\beta\in\Scal{L}{R}$, the components of the point packing $L$ need not contain the same number of components of $\beta RL$.
		We see this in the next example. 
		
		\begin{ex} \label{not-bijection1}
			Let $L = \G \cup (\frac{1}{2}+\G)$ be the $1 \times 2$ rectangular lattice viewed as a point packing that is generated by the square lattice $\G=\Z[i]$. 
			Take $\beta = 2\sqrt{2}$ and $R$ to be the rotation about the origin by $\frac{\pi}{4}$ in the counterclockwise direction.
			Then $R$ is a similarity isometry of $L$ with $\beta \in \Scal{L}{R}$, and $\beta RL$ is a similar subpacking of $L$.
			Figure~\ref{not-bijection-square} shows that $\beta R\G\subseteq\G$, and so $n=[\beta R\G : (\G \cap \beta R\G)]=1$.
			It follows from Theorem~\ref{component2} that each component of $\beta RL$ intersects exactly one component of $L$.
			However, both components of $\beta RL$, namely $\beta R\G$ and $\beta R(\frac{1}{2}+\G)$, are contained in $\G$,
			while $\frac{1}{2}+\G$ does not contain any component of $\beta RL$.	
			
			\begin{figure}[!ht]
				\centering
				\includegraphics[width=5.5 cm]{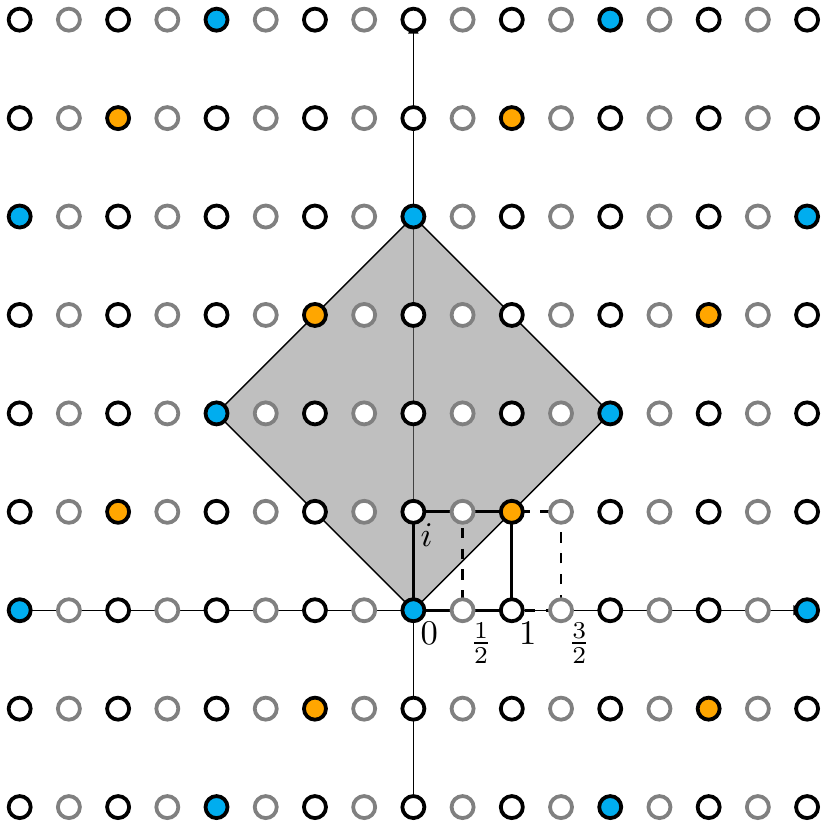}
				\caption{Let $L$ be the $1 \times 2$ rectangular lattice viewed as the point packing given by the union of the square lattice $\G = \Z[i]$ (black dots) 
					and its shifted copy $\frac{1}{2}+\G$ (gray dots).
					If $\beta = 2\sqrt{2}$ and $R$ is the rotation about the origin by $\frac{\pi}{4}$, 
					then the set $\beta RL$, consisting of the union of $\beta R\G$ (blue dots) and $\beta R(\frac{1}{2}+\G)$ (yellow dots), is a similar subpacking of $L$.} 
				\label{not-bijection-square}
			\end{figure}		
		\end{ex} 
		
		Let $R\in\os(L)$ and $\beta\in\Scal{L}{R}$. 
		Similar to the situation in Example~\ref{not-bijection1}, we look at what happens when $\beta\in\Scal{\G}{R}$, that is, when $\beta R\G\subseteq \G$.
		Here, we have $n=[\beta R\G : (\G \cap \beta R\G)]=1$, and so every component of $\beta RL$ is a subset of exactly one component of $L$.
		Theorem 3.2 now reduces to the following corollary.
		
		\begin{cor} \label{component2b-cor}
			Let $L = \bigcup_{k=0}^{m-1} (x_k + \G)$ be a point packing generated by a lattice $\G \subseteq \R^d$, $R \in \os(\G)$, and $\beta \in \Scal{\G}{R}$.
			Then $R\in\os(L)$ and $\beta\in\Scal{L}{R}$ if and only if 
			for every $k \in \{0,1,\ldots,m-1\}$, there exists a $j \in \{0,1,\ldots,m-1\}$ such that 
			$\beta R x_k - x_j \in \G$. 
		\end{cor}

		In the next section, we discuss planar examples that fall under the case above.
		
		Let us now introduce the following notation to help us identify the components of the similar subpacking $\beta RL$ of the 
		point packing $L=\bigcup_{k=0}^{m-1} (x_k + \G)$.
		The components of $\beta RL$ are determined by 
		identifying the pairs $(x_k, x_j)$ of shift vectors of $L$ satisfying the conditions in Theorem~\ref{component2}.
		The set of all such pairs $(x_k, x_j)$ will be denoted by $\tau(\beta RL)$, and is given by 
		\[\tau(\beta RL)=\{ (x_k, x_j) \in V_L \times V_L : \beta Rx_k - x_j \in \G + \beta R\G \}.\]
		Note that $|\tau(\beta RL)| = mn$ where $n = [\beta R\G : (\G \cap \beta R\G)] \leq m$. 
		In particular, if $\beta \in \Scal{\G}{R}$, then we see from Corollary~\ref{component2b-cor} that
		\[\tau(\beta RL) = \{ (x_k, x_j) \in V_L \times V_L : \beta Rx_k-x_j \in \G \}\]
		where $|\tau(\beta RL)| = m$. 
		
		Finally, let us examine the set of scaling factors $\Scal{L}{R}$ for a similarity isometry $R$ of $L$.
		Note that in general, negative scaling factors may exist, that is, it is possible that $\beta RL$ is a similar subpacking of $L$ for some $\beta <0$.
		In fact, for a lattice $\G$, if $\beta \in \R^+$ such that $\beta R\G \subseteq \G$, then $-\beta R\G \subseteq \G$.
		If a point packing $L$ is invariant under inversion, that is, $-L=L$, then for all $R \in \os(L)$ such that $0 < \beta \in \Scal{L}{R}$, we have $-\beta RL \subseteq L$. 
		This means that we can consider $-\beta < 0$ as a scaling factor of $L$. 
		
		Recall that for any similarity isometry $R$ of a point packing $L$ generated by a lattice~$\G$, 
		if $\beta\in\Scal{L}{R}$ and $n=[\beta R\G : (\G \cap \beta R\G)] \leq m$, then $\beta\in\frac{1}{n}\Scal{\G}{R}=\tfrac{1}{n}\den{\G}{R}\, \Z$.
		In words, any element of $\Scal{L}{R}$ may be written as an integral multiple of $\frac{1}{n}\den{\G}{R}$, for some $n\leq m$.
		This fact, together with Theorem~\ref{component2}, allows us to compute $\Scal{L}{R}$ for a given similarity isometry $R \in \os(L)$. 
		
		A discussion on the algebraic structure of the set $\os(L)$ can be found in Appendix~\ref{app}, where we show that $\os(L)$ is a monoid under certain assumptions. 
			
	\section{Planar Examples}
		We now apply the results in the previous section to planar point packings.
		In particular, we compute $\sos(L)$ and $\os(L)$ for certain point packings $L$ generated by the square lattice $\Z[i]$ or hexagonal lattice $\Z[\w]$ with $\w = e^{2\pi i/3}$, 
		and show that they are groups.
		Afterwards, we investigate two concrete examples: the $1 \times 2$ rectangular lattice viewed as a point packing generated by the square lattice, 
		and the hexagonal packing viewed as a point packing generated by the hexagonal lattice.	
		For both examples, we will show that each component of the similar subpacking is contained in a unique component of the point packing. 
		
		It is known that any similarity rotation $R$ of $\Z[i]$ (or $\Z[\w]$) corresponds to multiplication by a non-zero element of the form $z/|z|$, 
		where $z \in \Z[i]$ (or $\Z[\w]$)~\cite{BM98,G11,GB08,BSZ11}. 
		In addition, $\den{\Z[i]}{R} = \den{\Z[\w]}{R} = |z|$.
		Meanwhile, any reflection $T$ can be expressed as a product $T = R T_r$, where $R$ is a rotation, 
		and $T_r$ denotes the reflection about the real axis and corresponds to complex conjugation.
		Because of this, similarity rotations will be computed first in each example, followed by similarity reflections. 
		Moreover, a reflection $T$ will be represented in Tables~\ref{table1},~\ref{table3}, and~\ref{table5} by its rotational part $R$ in the representation $T = RT_r$. 

		Let us now proceed to compute for $\sos(L)$ and $\os(L)$ of point packings $L$ that are generated by a square or a hexagonal lattice, 
		all of whose shift vectors have rational components.
	
		\begin{prop} \label{os-planar-rational}
			Let $L = \bigcup_{k=0}^{m-1} (x_k + \G)$ be a point packing generated by the lattice $\G = \Z[i]$ (or $\Z[\w]$ with $\w = e^{2\pi i/3}$).
			If $x_k \in \Q(i)$ (or $\Q(\w)$) for all $k\in\{0,1,\ldots,m-1\}$, then $\sos(L) = \sos(\G)$ and $\os(L) = \os(\G)$.
		\end{prop}
		
		\begin{proof}
			Let $\G = \Z[i]$ and $R\in\sos(\G)$ corresponding to multiplication by $z/|z|$, where $0 \neq z \in \Z[i]$.
			For each $x_k \in V_L$, write $x_k = \frac{y_k}{r_k}$ for some $y_k \in \Z[i]$ and $0 \neq r_k \in \Z$.
			Take $\beta = \lcm(r_1,\ldots,r_{m-1}) \den{\G}{R} \in \Scal{\G}{R}$.
			It follows that for all $x_k \in V_L$, 
			we have \[\beta R x_k - x_0 = \lcm(r_1,\ldots,r_{m-1}) |z| \cdot \frac{z}{|z|} \cdot \frac{y_k}{r_k} - 0 = q_k z y_k \in \G,\] for some $q_k \in \Z$.
			Hence, $R \in \sos(L)$ by Corollary~\ref{component2b-cor}. Therefore, $\sos(L) = \sos(\G)$.
			
			For reflections $T = RT_r \in \os(L) \setminus \sos(L)$, note that $\overline{x_k} = \frac{\overline{y_k}}{r_k}$.
			Using the same scaling factor $\beta$, we obtain $\beta T x_k - x_0 = \beta R \overline{x_k} - x_0 = q_k z \overline{y_k} \in \G$, for some $q_k\in\Z$.
			This shows that $T \in \os(L)$ and therefore, $\os(L) = \os(\G)$.
			The proof for $\G = \Z[\w]$ is similar.
		\end{proof} 
		
		Hence, we obtain that $\os(L)$ and $\sos(L)$ are groups for a point packing $L$ generated by a square or a hexagonal lattice, all of whose shift vectors have rational components.
	
		We now look at two specific planar examples: the $1 \times 2$ rectangular lattice and the hexagonal packing.
	
		\subsection{The $\mathbf{1 \times 2}$ rectangular lattice.}
			Let $L = \G \cup (\frac{1}{2}+\G)$  be the $1 \times 2$ rectangular lattice viewed as a point packing that is generated by the square lattice $\G=\Z[i]$. 
			Suppose $R \in \os(L)$ and $\beta \in \Scal{L}{R}$.
			If $R$ is a similarity rotation of $L$, then by Theorem~\ref{component2}, $R$ corresponds to multiplication by $z/|z|$ for some $0 \neq z \in \Z[i]$ since $R\in \sos(\G)$, and 
			$\beta = (p/q)|z|$ for some relatively prime $p,q \in \Z \setminus \{0\}$ since $\beta \in \scal{\G}{R} = \den{\G}{R}\, \Q^{\ast}$.
			In addition, if $n = [\beta R\G : (\G \cap \beta R\G)]$, then $n \leq 2$ because $L$ has only two components.
			We have $\beta R\G = (pz/q)\G$ and 
			\[ \G \cap \beta R\G = \tfrac{1}{q} (q\G \cap pz\G) = \tfrac{1}{q} \lcm(pz,q) \G = \tfrac{pz}{\gcd(z,q)} \G. \]
			Here, we take the $\gcd$ and $\lcm$ of elements of the Euclidean ring $\Z[i]$. 
			Interpreting $n$ as the quotient of the volume of the fundamental domain of $\G \cap \beta R\G$ by the volume of the fundamental domain of $\beta R\G$, we obtain
			\[ n = \frac{p^2{|z|}^2}{{|\gcd(z,q)|}^2} \div \frac{p^2{|z|}^2}{q^2} = \frac{q^2}{{|\gcd(z,q)|}^2}. \]
			Suppose $n=2$.
			Then $q^2 = 2{|\gcd(z,q)|}^2$. 
			In addition, it follows from Theorem~\ref{component2} that 
			\[-\tfrac{1}{2}=\beta R(0) - \tfrac{1}{2} \in \G + \beta R\G= \tfrac{1}{q} (q\G + pz\G) = \tfrac{1}{q} \gcd(z,q) \G.\] 
			Thus, $q/[2 \gcd(z,q)] \in \G$.
			However, its number-theoretic norm is given by 
			\[\tfrac{q^2}{4{|\gcd(z,q)|}^2} = \tfrac{1}{2},\]
			and we reach a contradiction.
			Hence, $n=1$, that is, $\beta\in\Scal{\G}{R}$ or that each component of the similar subpacking $\beta RL$ is contained in a unique component of $L$.
			If $T=RT_r$ is a similarity reflection of $L$ where $R\in\sos(L)$, then similar arguments still yield that $n=1$. 
			
			Table~\ref{table1} summarizes the results when Corollary~\ref{component2b-cor} is applied.
			In Table~\ref{table1}, we list $\Scal{L}{R}$ and $\tau(\beta RL)$ for each $R \in \os(\G)$ corresponding to multiplication by \linebreak${(a+bi)/|a+bi|}$ 
			and depending on the parity of $a$ and $b$.
			Since $T_r L = L$, it follows that for any similarity reflection $T \in \os(L) \setminus \sos(L)$ and $\beta \in \Scal{L}{T}$, $\beta TL = \beta RT_r L = \beta RL$.
			Hence, the results in Table~\ref{table1} hold for any similarity isometry in $\os(L)$,
			where a reflection $T$ is represented in the table by its rotational part $R$ in the representation $T = RT_r$. 
			Note that when $R$ corresponds to multiplication by $a+bi \in \Z[i]$ where $a$ is odd and $b$ is even,
			then \[\Scal{L}{R}=\den{\G}{R}\,\Z=[\den{\G}{R}\, 2\Z]\cup[\den{\G}{R}(1+2\Z)],\]
			where $\beta R(\frac{1}{2}+\G) \subseteq \G$ if $\beta$ is an even multiple of $\den{\G}{R}$, 
			and $\beta R(\frac{1}{2}+\G) \subseteq \frac{1}{2}+\G$ if $\beta$ is an odd multiple of $\den{\G}{R}$.  
			
			Table~\ref{table1} implies that $\os(L) = \os(\G)$.
			This result also follows from Proposition~\ref{os-planar-rational} since the shift vectors $0$ and $\frac{1}{2}$ of $L$ are rational.
			The group $\os(L)$ may also be obtained by applying Corollary 3.7 of~\cite{G11} to $L$, now viewed as an ordinary lattice. 
			Nevertheless, viewing $L$ here as a point packing gives us a more concrete approach that is easy to describe geometrically. 
			It also allows us to extract exact information on the components of the similar subpackings $\beta RL$, 
			or in this case, the similar sublattices of the $1 \times 2$ rectangular lattice $L$.
			
			\begin{table}[!ht]
				\centering			
				\renewcommand{\arraystretch}{1.2}		
				\begin{tabular}{|c|c|c|}
					\hline
					$R: (a+bi)/|a+bi|$ & $\Scal{L}{R}$ & $\tau (\beta RL)$ \\
					\hline
					\multirow{2}{*}{$(a,b) \equiv (1,0) \!\! \pmod 2$} & $\den{\G}{R}\, 2\Z$ & $\{ (0,0), (\frac{1}{2},0) \}$ \\
					& $\den{\G}{R} (1+2\Z)$ & $\{ (0,0), (\frac{1}{2},\frac{1}{2}) \}$ \\
					\hline 
					$(a,b) \equiv (0,1), (1,1) \!\! \pmod 2$ & $\den{\G}{R}\, 2\Z$ & $\{ (0,0), (\frac{1}{2},0) \}$ \\
					\hline
				\end{tabular}
				\caption{Let $L$ be the $1 \times 2$ rectangular lattice viewed as a point packing generated by the square lattice $\G = \Z[i]$. 
					The sets $\Scal{L}{R}$ and $\tau (\beta RL)$ are given for all similarity isometries $R$ in $\os(L)$.
					Here, reflections $T = RT_r$ are represented by their rotational part $R$.}
				\label{table1}	
			\end{table}				
	
			\begin{ex}\label{1by2rect}
				Take $\beta = \sqrt{5}$ and $R$ to be the rotation about the origin by $\tan^{-1}(2)$ which corresponds to multiplication by $(1+2i)/\sqrt{5}$.
				Note that $\den{\G}{R} = \sqrt{5}$, and thus according to Table~\ref{table1}, 
				$R$ is a similarity isometry of $L$ with $\Scal{L}{R} = \sqrt{5} (2\Z) \ \cup \ \sqrt{5} (1+2\Z) = \sqrt{5}\, \Z$ and $\den{L}{R} = \sqrt{5}$. 
				Since $\beta = \sqrt{5}$, the second row of Table~\ref{table1} tells us that $\beta R\G \subseteq \G$ and $\beta R(\frac{1}{2}+\G) \subseteq \frac{1}{2}+\G$, 
				as can be verified in Figure~\ref{figure2}.
				
				\begin{figure}[!ht]
					\centering
					\includegraphics[width=5 cm]{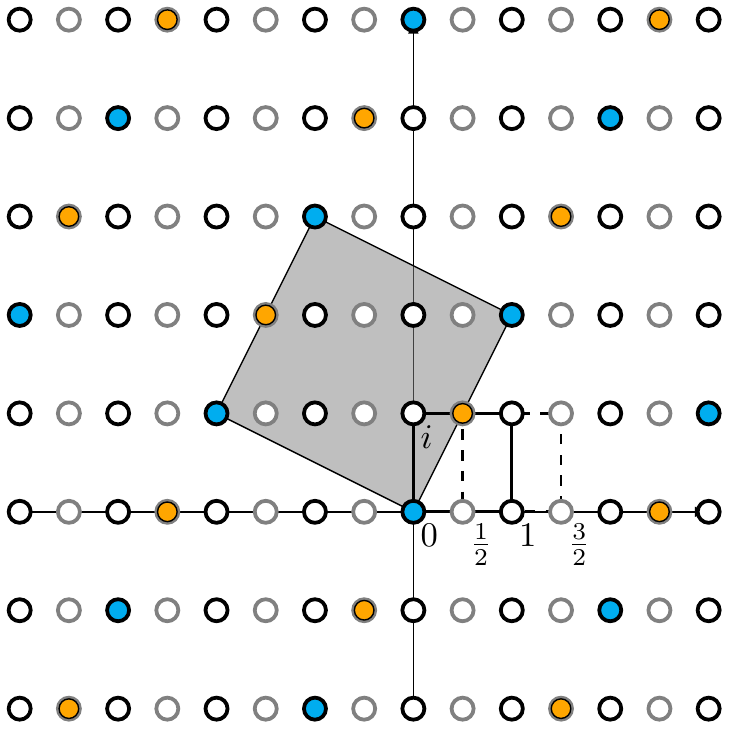}
					\caption{Let $L$ be the $1 \times 2$ rectangular lattice viewed as the union of the square lattice $\G = \Z[i]$ (black dots) and its shifted copy $\frac{1}{2}+\G$ (gray dots). 
						If $\beta = \sqrt{5}$, and $R$ corresponds to multiplication by $\frac{1+2i}{\sqrt{5}}$ (rotation about the origin by $\tan^{-1}(2)$),
						then the set $\beta RL$, consisting of the union of $\beta R\G$ (blue dots) and $\beta R(\frac{1}{2}+\G)$ (yellow dots), is a similar subpacking of $L$.}
					\label{figure2}
				\end{figure}
			\end{ex}
	
		\subsection{The hexagonal packing.}
			Let $L = \G \cup (\frac{2+\w}{3}+\G)$ be the hexagonal packing viewed as a point packing generated by the hexagonal lattice $\G=\Z[\w]$ with $\w=e^{2\pi i/3}$. 
			Note that $L$ is not a lattice.
			If $R \in \os(L)$ and $\beta \in \Scal{L}{R}$, then by Theorem~\ref{component2}, 
			$R\in\os(\G)$ and $\beta\in\scal{\G}{R}$. 
			Moreover, if $n = [\beta R\G : (\G \cap \beta R\G)]$, then $n \leq 2$ because $L$ has only two components.
			Since $\frac{2+\w}{3} - 0 = \frac{1}{3} (2+\w) \notin \frac{1}{2} \Z[\w]$, we obtain that $n\neq 2$ from Corollary~\ref{component2c}.
			Thus $n=1$, that is, $\beta\in\Scal{\G}{R}$ or that every component of the similar subpacking $\beta RL$ is contained in exactly one component of $L$. 
			
			Tables~\ref{table2} and \ref{table3} summarize the similarity isometries of $L$ obtained by applying Corollary~\ref{component2b-cor}.
			The results in Table~\ref{table2} hold for any similarity rotation $R \in \sos(L)$, 
			while those in Table~\ref{table3} hold for any similarity reflection $T \in \os(L) \setminus \sos(L)$. 
			Similar to Table~\ref{table1}, when the entry for $\Scal{L}{R}$ (or $\Scal{L}{T}$) consists of two lines, 
			it means that $\Scal{L}{R}$ (or $\Scal{L}{T}$) is the union of the sets in both lines. 
			The two tables together imply that $\os(L) = \os(\G)$, and thus, $\os(L)$ is a group. 
			This result is expected from Proposition~\ref{os-planar-rational}, since the shift vectors $0$ and $\frac{2+\w}{3}$ of $L$ are both in $\Q(\w)$.
			
			\begin{table}[!ht]
				\centering
				\renewcommand{\arraystretch}{1.2}
				\begin{tabular}{|c|c|c|}
					\hline
					$R: (a+b\w)/|a+b\w|$ & $\Scal{L}{R}$ & $\tau (\beta RL)$ \\
					\hline
					\multirow{2}{*}{$a+b \equiv 1 \!\! \pmod 3$} & $\den{\G}{R}\, 3\Z$ & $\{ (0,0), (\frac{2+\w}{3},0) \}$ \\
					& $\den{\G}{R} (1+3\Z)$ & $\{ (0,0), (\frac{2+\w}{3},\frac{2+\w}{3}) \}$ \\
					\hline 
					\multirow{2}{*}{$a+b \equiv 2 \!\! \pmod 3$} & $\den{\G}{R}\, 3\Z$ & $\{ (0,0), (\frac{2+\w}{3},0) \}$ \\
					& $\den{\G}{R} (2+3\Z)$ & $\{ (0,0), (\frac{2+\w}{3},\frac{2+\w}{3}) \}$ \\
					\hline 
					$a+b \equiv 0 \!\! \pmod 3$ & $\den{\G}{R}\, \Z$ & $\{ (0,0), (\frac{2+\w}{3},0) \}$ \\
					\hline
				\end{tabular}
				\caption{Let $L$ be the hexagonal packing viewed as a point packing generated by the hexagonal lattice $\G = \Z[\w]$, where $\w=e^{2\pi i/3}$. 
					The sets $\Scal{L}{R}$ and $\tau (\beta RL)$ are given for all similarity rotations $R$ in $\sos(L)$.}
				\label{table2}
			\end{table}
			
			\begin{table}[!ht]
				\centering
				\renewcommand{\arraystretch}{1.2}
				\begin{tabular}{|c|c|c|}
					\hline
					$T = R T_r \text{ where } R : (a+b\w)/|a+b\w|$ & $\Scal{L}{T}$ & $\tau (\beta TL)$ \\
					\hline
					\multirow{2}{*}{$a+b \equiv 1 \!\! \pmod 3$} & $\den{\G}{R}\, 3\Z$ & $\{ (0,0), (\frac{2+\w}{3},0) \}$ \\
					& $\den{\G}{R} (2+3\Z)$ & $\{ (0,0), (\frac{2+\w}{3},\frac{2+\w}{3}) \}$ \\
					\hline 
					\multirow{2}{*}{$a+b \equiv 2 \!\! \pmod 3$} & $\den{\G}{R}\, 3\Z$ & $\{ (0,0), (\frac{2+\w}{3},0) \}$ \\
					& $\den{\G}{R} (1+3\Z)$ & $\{ (0,0), (\frac{2+\w}{3},\frac{2+\w}{3}) \}$ \\
					\hline 
					$a+b \equiv 0 \!\! \pmod 3$ & $\den{\G}{R}\, \Z$ & $\{ (0,0), (\frac{2+\w}{3},0) \}$ \\
					\hline
				\end{tabular}
				\caption{Let $L$ be the hexagonal packing viewed as a point packing generated by the hexagonal lattice $\G = \Z[\w]$, where $\w=e^{2\pi i/3}$. 
					The sets $\Scal{L}{T}$ and $\tau (\beta TL)$ are given for all similarity reflections $T$ in $\os(L) \setminus \sos(L)$.  }
				\label{table3}
			\end{table}
		
			\begin{ex}\label{ex3}
				Take $\beta = 2$ and $R$ to be the rotation about the origin by $\frac{\pi}{3}$ which corresponds to multiplication by $1+\w$.
				Here, $\den{\G}{R} = 1$, and thus according to Table~\ref{table2}, $R$ is a similarity isometry of $L$ with $\Scal{L}{R} = 3\Z \ \cup \ (2+3\Z)$ and $\den{L}{R} = 2$. 
				Since $\beta=2$, the fourth row of Table~\ref{table2} states that $\beta R\G \subseteq \G$ and $\beta R(\frac{2+\w}{3}+\G) \subseteq \frac{2+\w}{3}+\G$, 
				as can be verified in Figure~\ref{figure3}.
				
				\begin{figure}[!ht]
					\centering
					\includegraphics[width=5 cm]{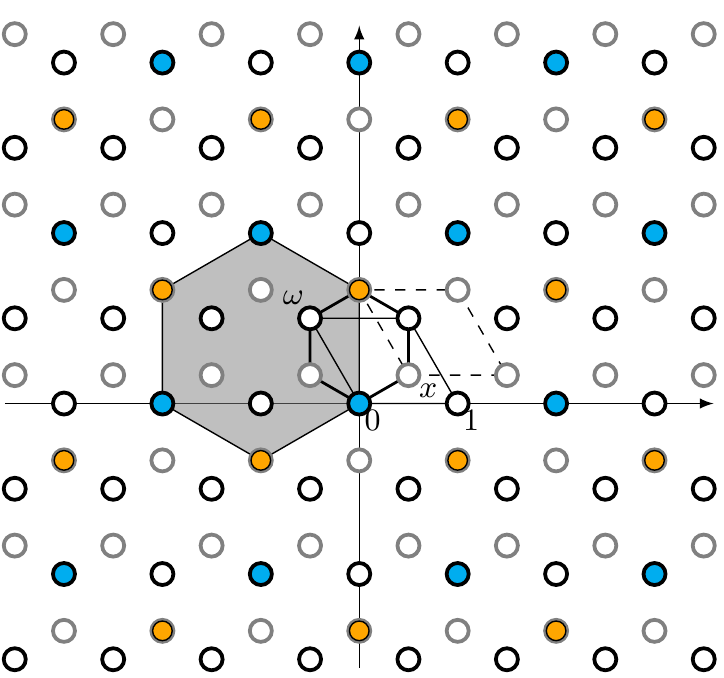}
					\caption{Let $L$ be the hexagonal packing viewed as the union of the hexagonal lattice $\G = \Z[\w]$ (black dots), where $\w=e^{2\pi i/3}$,  
						and its shifted copy $\frac{2+\w}{3}+\G$ (gray dots).
						If $\beta = 2$ and $R$ corresponds to multiplication by $1+\w$ (rotation about the origin by $\frac{\pi}{3}$),  
						then the set $\beta RL$, consisting of the union of $\beta R\G$ (blue dots) and $\beta R(\frac{2+\w}{3}+\G)$ (yellow dots), is a similar subpacking of $L$.}
					\label{figure3}
				\end{figure}	
			\end{ex}
		
	\section{Shifted Point Packings}
		Let us now consider rotations about points different from the origin. 
		Note that rotating a point packing $L \subseteq \R^d$ about a point $-x \in \R^d$ is equivalent to rotating $x+L$ about the origin.
		For instance, consider the hexagonal packing $L = \G \cup (x+\G)$, where $\G = \Z[\w]$ is the hexagonal lattice and $x=\frac{2+\w}{3}$.
		Observe that applying linear isometries to $L$ fixes the origin which is located at a vertex of a hexagon.
		To rotate about a center of a hexagon (a point of maximal symmetry of $L$), 
		say, $-x$, one may first translate all points of $L$ by the vector $x$ and afterwards rotate about the origin.
		This observation leads us to investigate translates of point packings, which will be referred to as \emph{shifted point packings}, and their similarity isometries.
		The same idea was employed in \cite{AGL14} to find the coincidence isometries of point packings about points different from the origin.
		
		If $L = \bigcup_{k=0}^{m-1} (x_k + \G)$ is a point packing generated by a lattice $\G$, and $x \in \R^d$, 
		then the shifted point packing $x+L$ is given by $x+L = \bigcup_{k=0}^{m-1} (x + x_k + \G)$.
		Observe that $x+L$ is still a union of shifted copies of $\G$ but having shift vectors $x+x_k$, that is, shift vectors $x_k$ of $L$ translated by $x$.
		The only difference is that if $x\notin\per(L)$, then the shifted point packing $x+L$ does not include $\G$ (and the origin).
		Hence, the usual definitions and notations regarding similarity isometries of point packings will be applied to shifted point packings.
		Moreover, the results on similarity isometries of point packings in Section~\ref{simpointpack} are applicable to shifted point packings.
		In particular, we use analogous results to Lemma~\ref{component1-lem}, Theorem~\ref{component2}, and Corollaries~\ref{component2c} and~\ref{component2b-cor} 
		to determine $\os(x+L)$ and $\Scal{x+L}{R}$ for a given similarity isometry $R \in \os(x+L)$.
		
		We illustrate these results by looking at the shifted hexagonal packing \[x+L = (\tfrac{2+\w}{3}+\G) \cup (\tfrac{4+2\w}{3}+\G),\] 
		where $L = \G \cup (\frac{2+\w}{3} + \G)$ is the hexagonal packing, $\G = \Z[\w]$, with $\w = e^{2\pi i/3}$, is the hexagonal lattice, and $x=\frac{2+\w}{3}$.
		Let $R \in \os(x+L)$ and $\beta \in \Scal{x+L}{R}$.
		It follows from the analogue of Theorem~\ref{component2} that $R\in\os(\G)$, $\beta\in\scal{\G}{R}$, 
		and $n=[\beta R\G : (\G \cap \beta R\G)]$ is at most $2$.
		Similar to what happened in the hexagonal packing, we conclude that $n=1$ from the analogue of Corollary~\ref{component2c}.
		Thus, $\beta\in\Scal{\G}{R}$ and each component of $\beta R(x+L)$ is contained in a unique component of $x+L$. 
		
		Tables~\ref{table4} and \ref{table5} were obtained by applying analogous results of Corollary~\ref{component2b-cor}.
		The results in Table~\ref{table4} hold for any similarity rotation $R \in \sos(x+L)$, 
		while those in Table~\ref{table5} hold for any similarity reflection $T \in \os(x+L) \setminus \sos(x+L)$. 
		In all cases, $\Scal{x+L}{S} = \den{\G}{S} (1+3\Z) \cup \den{\G}{S} (2+3\Z)$ for any similarity isometry $S$ of $x+L$. 
		The two tables together imply that $\os(x+L) \subsetneq \os(L) = \os(\G)$. 
		One may verify that $\os(x+L)$ forms a group, and hence, is a proper subgroup of $\os(\G)$.
		
		\begin{table}[!ht]
			\centering
			\renewcommand{\arraystretch}{1.3}
			\begin{tabular}{|c|c|c|}
				\hline
				$R: (a+b\w)/|a+b\w|$ & $\Scal{x+L}{R}$ & $\tau (\beta R(x+L))$ \\
				\hline
				\multirow{2}{*}{$a+b \equiv 1 \!\! \pmod 3$} & $\den{\G}{R} (1+3\Z)$ & $\{ (\frac{2+\w}{3},\frac{2+\w}{3}), (\frac{4+2\w}{3},\frac{4+2\w}{3}) \}$ \\
				& $\den{\G}{R} (2+3\Z)$ & $\{ (\frac{2+\w}{3},\frac{4+2\w}{3}), (\frac{4+2\w}{3},\frac{2+\w}{3}) \}$ \\
				\hline 
				\multirow{2}{*}{$a+b \equiv 2 \!\! \pmod 3$} & $\den{\G}{R} (1+3\Z)$ & $\{ (\frac{2+\w}{3},\frac{4+2\w}{3}), (\frac{4+2\w}{3},\frac{2+\w}{3}) \}$ \\
				& $\den{\G}{R} (2+3\Z)$ & $\{ (\frac{2+\w}{3},\frac{2+\w}{3}), (\frac{4+2\w}{3},\frac{4+2\w}{3}) \}$ \\
				\hline 
			\end{tabular}
			\caption{Let $x+L$ be the shifted hexagonal packing where $L$ is the hexagonal packing generated by the hexagonal lattice $\G = \Z[\w]$ 
				with $\w=e^{2\pi i/3}$, and $x=\frac{2+\w}{3}$. 
				The sets $\Scal{x+L}{R}$ and $\tau (\beta R(x+L))$ are given for all similarity rotations $R$ in $\sos(x+L)$.}
			\label{table4}
		\end{table}
		
		\begin{table}[!ht]
			\centering
			\renewcommand{\arraystretch}{1.3}
			\begin{tabular}{|c|c|c|}
				\hline
				$T = RT_r \text{ with } R: (a+b\w)/|a+b\w|$ & $\Scal{x+L}{T}$ & $\tau (\beta T(x+L))$ \\
				\hline
				\multirow{2}{*}{$a+b \equiv 1 \!\! \pmod 3$} & $\den{\G}{R} (1+3\Z)$ & $\{ (\frac{2+\w}{3},\frac{4+2\w}{3}), (\frac{4+2\w}{3},\frac{2+\w}{3}) \}$ \\
				& $\den{\G}{R} (2+3\Z)$ & $\{ (\frac{2+\w}{3},\frac{2+\w}{3}), (\frac{4+2\w}{3},\frac{4+2\w}{3}) \}$ \\
				\hline 
				\multirow{2}{*}{$a+b \equiv 2 \!\! \pmod 3$} & $\den{\G}{R} (1+3\Z)$ & $\{ (\frac{2+\w}{3},\frac{2+\w}{3}), (\frac{4+2\w}{3},\frac{4+2\w}{3}) \}$ \\
				& $\den{\G}{R} (2+3\Z)$ & $\{ (\frac{2+\w}{3},\frac{4+2\w}{3}), (\frac{4+2\w}{3},\frac{2+\w}{3}) \}$ \\
				\hline 
			\end{tabular}
			\caption{Let $x+L$ be the shifted hexagonal packing where $L$ is the hexagonal packing generated by the hexagonal lattice $\G = \Z[\w]$
				with $\w=e^{2\pi i/3}$, and $x=\frac{2+\w}{3}$. 
				The sets $\Scal{x+L}{T}$ and $\tau (\beta T(x+L))$ are given for all similarity reflections $T$ in $\os(x+L) \setminus \sos(x+L)$.}
			\label{table5}
		\end{table}
	
		\begin{ex}\label{ex4}
			Take $\beta = 1$ and $R$ to be the rotation about the origin by $\frac{\pi}{3}$ which corresponds to multiplication by $1+\w$.
			Note that $\den{\G}{R} = 1$, and thus according to Table~\ref{table4}, 
			$R$~is a similarity isometry of $x+L$ with $\Scal{x+L}{R} = (1+3\Z) \ \cup \ (2+3\Z)$ and ${\den{x+L}{R} = 1}$. 
			Since $\beta = 1$, the third row of Table~\ref{table4} states that $\beta R(\frac{2+\w}{3}+\G) \subseteq \frac{4+2\w}{3}+\G$ 
			and $\beta R(\frac{4+2\w}{3}+\G) \subseteq \frac{2+\w}{3}+\G$, as can be verified in Figure~\ref{figure4}.
			In fact, Figure~\ref{figure4} shows that $\beta R(\frac{2+\w}{3}+L) = \frac{2+\w}{3}+L$, 
			since $\beta R(\frac{2+\w}{3}+\G) = \frac{4+2\w}{3}+\G$ and $\beta R(\frac{4+2\w}{3}+\G) = \frac{2+\w}{3}+\G$.
			This is expected since the hexagonal packing is symmetric with respect to the $\frac{\pi}{3}$ rotation about the center of the hexagons.
			
			\begin{figure}[!ht]
				\centering
				\includegraphics[width=5 cm]{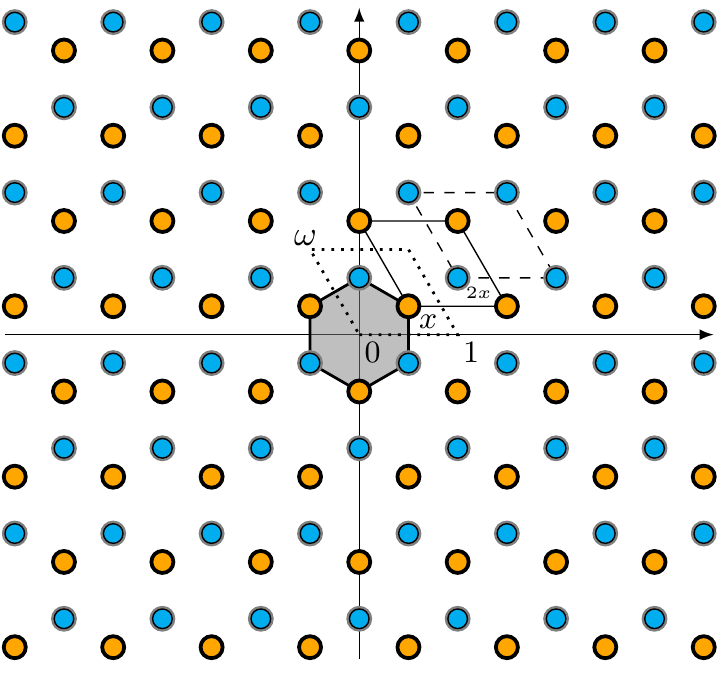}
				\caption{Let $x+L$ be the shifted hexagonal packing given by the union of $\frac{2+\w}{3}+\G$ (black outline) and $\frac{4+2\w}{3}+\G$ (gray outline), 
					where $x = \frac{2+\w}{3}$, $L$ is the hexagonal packing $\G \cup (\frac{2+\w}{3}+\G)$, and $\G$ is the hexagonal lattice $\Z[\w]$ with $\w=e^{2\pi i/3}$.  
					If $\beta = 1$ and $R$ corresponds to multiplication by $1+\w$ (rotation about the origin by $\frac{\pi}{3}$),
					then the set $\beta R(x+L)$, consisting of the union of $\beta R(\frac{2+\w}{3}+\G)$ (blue dots) and $\beta R(\frac{4+2\w}{3}+\G)$ (yellow dots), is a similar subpacking of $x+L$. 
					Here, we see that $\beta R(x+L) = x+L$.}
				\label{figure4}
			\end{figure}
		\end{ex}
			
		We deduce from Tables~\ref{table4} and \ref{table5} that $\os(x+L)$ is a proper subset of $\os(L) = \os(\G)$.
		Consequently, there are similarity isometries of $L$ which are not similarity isometries of $x+L$. 
		In other words, there are more similarity isometries of $L$ when one rotates $L$ about a vertex of a hexagon than when one rotates $L$ about a center of a hexagon.
		This is a surprising result, since the hexagonal packing has more symmetries about the center of a hexagon than about a vertex of a hexagon.
		
		Moreover, observe that if $R \in \os(x+L)$, then the sets of scaling factors $\Scal{x+L}{R}$ and $\Scal{L}{R}$ vary.
		For instance, consider Figures~\ref{figure3} and \ref{figure4}. 
		Even if the rotation $R$ about the origin by $\frac{\pi}{3}$ is a similarity isometry of both $x+L$ and $L$, 
		the sets of scaling factors obtained are different: $\Scal{L}{R} = 3\Z \ \cup \ (2+3\Z)$ and $\Scal{x+L}{R} = (1+3\Z) \ \cup \ (2+3\Z)$. 
		Hence, changing the rotation point (from a vertex to a center of a hexagon) affects not only the set of similarity isometries, but also the set of scaling factors and the set $\tau$ that describes each component of the similar subpacking as a subset of some component of the point packing.
	
	\section{Outlook}
	
		The notions of a similarity isometry of a lattice and a similar sublattice were extended to (crystallographic) point packings. 
		Equivalent conditions were obtained to characterize a similarity isometry of a point packing and its set of scaling factors.
		The approach employed here was to examine the components of the similar subpacking component-wise: 
		each component of the similar subpacking intersects exactly $n$ components of the point packing.
		These results were applied to planar point packings whose generating lattice is the square lattice $\Z[i]$ or the hexagonal lattice $\Z[\w]$.
		Particular planar examples discussed, namely the $1 \times 2$ rectangular lattice and the hexagonal packing, 
		fall under the case $n=1$ where each component of the similar subpacking is contained in a unique component of the point packing. 
		It would be interesting to find more examples where $n>1$, such as Example~\ref{ex-component}, 
		and study specific examples such as point packings associated to Archimedean tilings~\cite{EFL,SY19}. 
		
		The set $\os(L)$ of similarity isometries of a point packing $L$ is a subset of the group of similarity isometries of the generating lattice $\G$ of $L$.
		Moreover, the set $\Scal{L}{R}$ of scaling factors with respect to $L$ is a subset of the 
		set $\scal{\G}{R}$ of scaling factors with respect to~$\G$.
		We have also shown in Appendix~\ref{app} that $\os(L)$ is a monoid under certain assumptions.
		Although the question of whether $\os(L)$ forms a group remains unsolved, 
		we have shown that $\os(L)$ is a group for point packings generated by a square or a hexagonal lattice, all of whose shift vectors have rational components.
		Our initial computations have also shown that point packings generated by a square or a hexagonal lattice with irrational shift vectors 
		have similarity isometries which still form a group.
		
		The framework and the methods presented in this study offer a more concrete approach in studying similarity isometries of ideal crystals that is easy to describe geometrically. 
		Since point packings serve as models for crystals having multiple atoms per primitive unit cell, 
		it would be interesting to investigate similarity isometries of other point packings, particularly three-dimensional ones that correspond to models of actual ideal crystals. 
		Similarity isometries of the diamond packing have already been identified by the authors and will be part of a future work. 
		
		Finally, the notion of a shifted point packing was introduced in the interest of examining similarity isometries of a point packing about points different from the origin.
		This is because rotating a point packing about a point $-x$ is equivalent to rotating the translate of the point packing by $x$ about the origin. 
		In particular, we investigate a shifted hexagonal packing and its similarity isometries that correspond to similarity isometries of the hexagonal packing about a center of a hexagon.
		The shifted hexagonal packing has less similarity isometries than the hexagonal packing.
		This is unexpected since the hexagonal packing has more symmetries about the center of a hexagon than about a vertex of a hexagon.
		It is also notable that the set of similarity isometries of the shifted hexagonal packing forms a proper subgroup of the set of similarity isometries of the hexagonal packing.
	
	\appendix	
	\section{Algebraic Structure of $\os(L)$}\label{app}
	
		We now look at the algebraic structure of $\os(L)$.
		The identity isometry $R=\mathbbm{1}$ is clearly in $\os(L)$.
		The following result shows a sufficient condition for $\os(L)$ to be closed under composition.
		
		\begin{prop}\label{monoid}
			Let $L = \bigcup_{k=0}^{m-1} (x_k + \G)$ be a point packing generated by a lattice $\G \subseteq \R^d$. 
			If $\Scal{L}{R} \subseteq \Scal{\G}{R}$ for all $R \in \os(L)$, then $\os(L)$ is closed under composition. 
		\end{prop}
		\begin{proof}
			Let $R_1, R_2 \in \os(L)$ and $x_i \in V_L$.
			By Corollary~\ref{component2b-cor}, $R_1 \in \os(\G)$ and there exist $\alpha \in \Scal{\G}{R_1}$ 
			and $x_j \in V_L$ such that $\alpha R_1 x_i - x_j \in \G$.
			Similarly, $R_2 \in \os(\G)$ and there exist $\beta \in\Scal{\G}{R_2}$ and $x_k \in V_L$ such that $\beta R_2 x_j - x_k \in \G$.
			Note that $R_2 R_1$ is an element of the group $\os(\G)$, and $\alpha \beta\in\Scal{\G}{R_2}\Scal{\G}{R_1}\subseteq\Scal{\G}{R_2 R_1}$.
			In addition, $\alpha \beta R_2 R_1 x_i - \beta R_2 x_j \in \beta R_2 \G \subseteq \G$, which implies that $\alpha \beta R_2 R_1 x_i - x_k\in \G$.
			It now follows from Corollary~\ref{component2b-cor} that $R_2R_1\in\os(L)$.
		\end{proof}
		
		The converse of Proposition~\ref{monoid} does not hold, that is, 
		closure of $\os(L)$ under composition does not imply that $\Scal{L}{R} \subseteq \Scal{\G}{R}$ for all $R \in \os(L)$. 
		Indeed, recall from Example~\ref{ex-component} that $L$ is a lattice. Thus, $\os(L)$ is a group and hence is closed under composition.
		However, if $R$ is again the rotation about the origin by $\frac{\pi}{2}$ in the counterclockwise direction, then
		$\Scal{L}{R} = \den{L}{R}\, \Z = \Z$ is not a subset of $\Scal{\G}{R} = \den{\G}{R}\, \Z = 3\Z$. 
		
		Proposition~\ref{monoid} assures us that $\os(L)$ is at least a monoid whenever the assumption $\Scal{L}{R} \subseteq \Scal{\G}{R}$ for all $R \in \os(L)$ holds.
		The closure of $\os(L)$ under inverses, and hence whether $\os(L)$ forms a group, remains unsolved.
		Nonetheless, there are many examples of planar point packings $L$ for which $\os(L)$ is a group, and these were discussed in Section~4.
	
	\section*{Acknowledgements}
		This work was funded by the UP System Enhanced Creative Work and Research Grant (ECWRG 2018-1-007).
		The authors also thank P.~Zeiner for pointing out a critical error in a previous version of this manuscript.
	
	\bibliography{ref}
	\bibliographystyle{amsplain}
\end{document}